\documentclass{amsart}
\usepackage{graphicx,xcolor,tikz,tikz-cd}
\usepackage{amsmath}
\usepackage{amssymb}
\usepackage{amsthm}

\usepackage{fullpage}

\makeatletter
\@namedef{subjclassname@2020}{%
  \textup{2020} Mathematics Subject Classification}
\makeatother

\newtheorem{theorem}{Theorem}[section]
\newtheorem{corollary}[theorem]{Corollary}
\newtheorem{lemma}[theorem]{Lemma}
\newtheorem{proposition}[theorem]{Proposition}
\theoremstyle{definition}

\theoremstyle{remark}
\newtheorem{rem}[theorem]{Remark}
\theoremstyle{definition}
\newtheorem{example}[theorem]{Example}
\numberwithin{equation}{section}

\newcommand{\set}[1]{\left\{#1\right\}}

\newcommand{\R}{\mathbb R}

\newcommand{\N}{\mathbb N}

\newcommand{\C}{\mathbb C}
\newcommand{\eps}{\varepsilon}

\newcommand{\CH}{{\mathcal H}}

\newcommand{\CL}{{\mathcal L}}
\newcommand{\CM}{{\mathcal M}}
\newcommand{\CN}{{\mathcal N}}
\newcommand{\CP}{{\mathcal P}}
\newcommand{\CQ}{{\mathcal Q}}

\newcommand{\CW}{{W}}

\newcommand{\interior}[1]{{\kern0pt#1}^{\mathrm{o}}}

\begin{document}

\title[]{On the escape rate for intermittent maps with holes shrinking around the indifferent fixed point}
\author{Claudio Bonanno}
\address{Dipartimento di Matematica, Universit\`a di Pisa, Largo Bruno Pontecorvo 5, 56127 Pisa, Italy}
\email{claudio.bonanno@unipi.it}
\author{Sharvari Neetin Tikekar}
\address{Dipartimento di Matematica, Universit\`a di Pisa, Largo Bruno Pontecorvo 5, 56127 Pisa, Italy}
\email{sharvari.tikekar@gmail.com}

\begin{abstract}
We study non-uniformly expanding maps of the unit interval with a parabolic fixed point at the origin that admit an ergodic absolutely continuous invariant measure, which may be finite or infinite. By introducing a hole defined by an interval containing the parabolic fixed point, we analyze the escape rate of the resulting open system and its asymptotic behavior as the hole shrinks. Our approach relies on the transfer operator associated with the dynamical system and on the relationship between the transfer operators of the original system and its induced version. The results extend to this general framework previous investigations which considered special cases.
\end{abstract}

\subjclass[2020]{37E05, 37D25, 37D35, 37A05}
\keywords{Intermittent maps; open dynamical systems; escape rate; shrinking holes}
\thanks{The authors are partially supported by 
the PRIN Grant 2022NTKXCX ``Stochastic properties of dynamical systems'' funded by the Ministry of University and Research, Italy. The authors acknowledge the MUR Excellence Department Project awarded to the Department of Mathematics, University of Pisa, CUP I57G22000700001. This research is part of the authors' activity within the UMI Group ``DinAmicI'' \texttt{www.dinamici.org} and of CB's activity within the Gruppo Nazionale di Fisica Matematica, INdAM, Italy.}

\maketitle

\section{Introduction} \label{sec:intro}

Given an ergodic conservative measure-preserving transformation $F$ of a $\sigma$-finite measure space $(X,\mu)$, a \emph{hole} of the associated dynamical system $(X,\mu,F)$ is a measurable set $H\subset X$ with $\mu(H)>0$ such that when $F^k(x) \in H$ the orbit of $x$ escapes from $X$. The investigation of properties of open system is more than forty years old as the first attempts can be traced back to the late '70s. See  \cite{PY79,FKMP95} and references therein for the first approaches to the problem. This kind of systems has been often used to model different situations of interest for the  physics community (see e.g. \cite{APT13}). 

Let $S_n$ be the set of points which do not escape up to time $n$, and define the \emph{survival probability at time $n$} as $\mu(S_n)$. The sequence $\mu(S_n)$ is decreasing and vanishing as $n\to \infty$, and one can study its exponential rate of convergence, that is 
\[
\gamma_\mu(H) := \limsup_{n\to \infty}\, \frac{-\log \mu(S_n)}{n}\, ,
\]
which is called the \emph{escape rate of $H$ with respect to $\mu$}. In this paper, we are in particular interested in studying the asymptotic behaviour of the escape rate as the hole shrinks to a point.

The escape rate has been studied mainly for hyperbolic systems, we refer to \cite{DY} for references, and rigorous results on its asymptotic behaviour as the hole shrinks are given in \cite{KL} and \cite{PU} by a functional analytic approach. However, it is well known that the statistical properties of dynamical systems can change dramatically when one moves from uniformly hyperbolic systems to \emph{intermittent} ones. These are systems whose orbits alternate between regular and chaotic behaviour, with the ``degree'' of intermittency determined by the duration of the regular phases. The terminology originates from the work of Pomeau and Manneville \cite{PM} in the mathematical physics literature, where such systems were introduced as a simple model for the physical phenomenon of alternation between turbulent and laminar phases in a fluid. The Pomeau-Manneville system is defined as a map of the interval $[0,1]$, in which intermittency arises from the presence of a parabolic fixed point at the origin, while the map is expanding elsewhere. The regular, or laminar, phases of the orbits then correspond to the long time spans spent near the origin.

In this paper, we study parabolic maps of the interval, namely maps sharing the same qualitative features as the Pomeau–Manneville maps. We introduce a hole $H$ and investigate the asymptotic behaviour of the associated escape rate as the hole shrinks to a point. This problem has been considered previously in several settings. In \cite{DF} and \cite{DT}, the hole $H$ is placed far from the parabolic fixed point at the origin, and it is shown that in this case the escape rate decays only at a polynomial rate as the hole shrinks. The case in which the hole contains the origin and shrinks around it is studied in \cite{FMS11}, under assumptions on the degree of intermittency ensuring the existence of an invariant probability measure equivalent to Lebesgue measure (see assumption (H5) in Section~\ref{sec:maps}), and in \cite{MK}, where a family of piecewise linear parabolic maps isomorphic to a countable Markov chain is considered (see Example~\ref{ex:pl} below). A closely related setting is examined in \cite{BC}, which focuses on the Farey map, a particular example of a parabolic map preserving an infinite absolutely continuous measure (see Example~\ref{examples} below), and considers a family of `generalised'' holes shrinking around the origin. In that work, standard holes could not be treated, as the method required the system with a hole to satisfy a form of analyticity up to the origin. Finally, families of intermittent maps with holes are also investigated in the recent paper \cite{BK}, where various ergodic properties of the resulting open systems are studied.

In this paper, we consider the case of a family of holes $H$ containing the origin and shrinking around it for parabolic maps of the interval with any degree of intermittency. The exact asymptotic behaviour of the escape rate is obtained for \emph{Markov holes} (see Section \ref{sec:open}) in Theorem \ref{thm:shrinking}, extending to this case the results in \cite{MK,BC}. To conclude, in Corollary \ref{general-hole}, we prove that the asymptotic behaviour of the escape rate for a general hole of the form $[0,\eps]$, with $\eps>0$, is exactly controlled by that of Markov holes. The method of proofs uses the \emph{transfer operator} of a dynamical system and its spectral properties, together with the standard inducing procedure. The main result is Theorem \ref{thm:escape_H_hatH}, which gives an exact relation between the escape rates of a parabolic map and of its induced version (see \eqref{aim-1}). Then, Proposition \ref{prop-utile} is the fundamental step to obtain the asymptotic behaviour of the escape rate of the parabolic map when the hole shrinks around the parabolic fixed point.

The paper is organized as follows. In Section \ref{sec:maps}, we introduce the parabolic maps we consider and their induced version. Then, we recall their symbolic representations and the notion of transfer operators, together with the relations with the existence of invariant measures and their spectral properties. In Section \ref{sec:open}, we introduce the open systems and give the main results.

\section{The maps} \label{sec:maps}
We consider \emph{parabolic maps} $F:[0,1]\to [0,1]$, that is maps satisfying the following properties.
\begin{itemize}
\item[(H1)] $F$ is \emph{piecewise monotone} with respect to a partition $J=\set{J_0,J_1}$ of the interval with $J_{0} :=[0,a]$ and $J_{1} := (a,1]$, for some $a\in (0,1)$, and has \emph{full branches}, that is $F(J_0)=\overline{F(J_1)}=[0,1]$.
\item[(H2)] For $r\in \N \cup \{\infty\}$, $r\ge 1$, the restriction $F_0:=F|_{J_0}$ is $C^r$ on $(0,a)$ and extends to a $C^r$ function on $(0,a]$, $F_1:=F|_{J_1}$ is $C^r$ on $(a,1)$ and extends to a $C^r$ function on $[a,1]$. We denote by $\phi_i := F_i^{-1}$ the local inverses of $F$.
\item[(H3)] $0$ is an \emph{indifferent fixed point}, that is $F(0)=0$ and $F'(x) -1 \sim c\, x^{s}$ as $x\to 0^+$, for some $c>0$ and $s > 0$ (note that $F_0$ is $C^r$ at $0$ only for $s \ge r-1$).
\item[(H4)] There exist $\eps>0$ and $\beta>1$ such that $F'(x)$ is increasing in $(0,\eps)$ and $|F'(x)|\ge \beta$ for all $x\ge \eps$.
\item[(H5)] $F$ admits a unique ergodic conservative invariant measure $\mu$ equivalent\footnote{That is, $\mu$ is absolutely continuous and its density is bounded away from 0.} to the Lebesgue measure $m$.
\end{itemize}

In the case $r\ge 2$, assumption (H5) is implied for example by the boundedness of $|F''|/|F'|^2$ on $[\eps,1)$ for some $\eps>0$. The invariant measure $\mu$ is finite if and only if $s<1$ and its density is bounded from above on all closed sub-intervals non-containing the indifferent fixed point (see \cite{Th1,Th2}).

\begin{example} \label{examples}
A classical example of a family of maps satisfying our assumptions are the so-called \emph{Pomeau-Manneville maps}, introduced and studied in \cite{M,PM} to investigate the phenomenon of intermittency in physics and defined as
\[
F_{pm}(x) := x + x^{1+s} \, \text{(mod $1$),} \quad \text{for $s>0$.}
\]
A somewhat simpler family of maps with the same dynamical properties have been introduced in \cite{LSV} and are now referred to as \emph{Liverani-Saussol-Vaienti maps}. The maps have been defined as
\[
F_{lsv}(x):=\begin{cases}
x(1+2^\alpha\, x^{s}), & \text{for $x\in [0, \frac 12]$,}\\[0.1cm]
2x-1, & \text{for $x\in (\frac 12, 1]$},
\end{cases} \quad \text{for $s>0$.}
\]
These families of maps, together with generalizations, are studied in \cite{BL} in the case $s\ge 1$ as models of infinite measure-preserving interval maps.

In addition, an important example of a parabolic map for $s=1$, which is related with the continued fraction expansion of real numbers, is the \emph{Farey map} defined as
\[
F_f(x) := \begin{cases}
        \frac{x}{1-x}, & \text{for $x\in [0, \frac 12]$,}\\[0.1cm]
        \frac{1-x}{x}, & \text{for $x\in [\frac 12, 1]$}.
    \end{cases}
\]
The Farey map satisfies our assumptions (H1), (H2), (H3), and (H5). Assumption (H4) is not satisfied since $F'_f(1)=-1$. However, the results in this paper apply to the Farey map by small changes.
\end{example}

By now, a standard way to study parabolic maps of the interval is inducing on the sets where the dynamics is expanding. Considering $(a,1]$ as the reference set, we define $\tau : [0,1] \rightarrow \N$ as,
\[
        \tau(x) := 1 + \inf \set{ i \ge 0 \, : \, F^{i}(x) \in (a,1]}, \quad x \in [0,1]. 
\]    
Note that $\tau -1$ is in fact the hitting time of a point to $(a,1]$ and for all $x > 0$, $\tau(x)$ is finite. Define the \emph{jump transformation} $G : [0,1] \to [0,1]$ of $F$ by,
\[
    G(x) := \begin{cases}
        F^{\tau(x)}(x), \ & \ \text{if} \ x \neq 0, \\
        0, \  & \ \text{if} \ x = 0.
    \end{cases}
\] 
Let $(a_{n})_{n \ge 0} \downarrow 0$ denote the following sequence of pre-images of $1$,
\begin{equation} \label{preimages}
        a_{0}:=1, \ \ \ a_{1}:= a, \ \ \ \text{and} \ \ \   a_{n}:=\phi_{0}(a_{n-1}), \  n \ge 2. 
\end{equation} 
For every $n\ge 1$, the interval $A_{n}:= (a_{n}, a_{n-1}] = \{ x \in [0,1] : \tau(x) = n \}$ is the level set of $\tau$. Observe that on the partition $\{ A_{n}\}_{n \ge 1}$ the jump transformation $G$ has countably many smooth monotone full branches, $\overline{G(A_n)}=[0,1]$, and it is expanding, that is $|G'(x)| \ge \beta$ for all $x\in \cup_n \interior{A}_n$, with $\interior{A}_n=(a_{n}, a_{n-1})$. In particular,
\begin{equation} \label{GAn}
G|_{A_n} = F_1 \circ F_0^{n-1}, \  n \ge 1, 
\end{equation}
and $G$ admits the local inverses 
\begin{equation} \label{GAn-inv}
\zeta_n := (G|_{\interior{A}_n})^{-1} : (0,1) \to \interior{A}_n, \quad \zeta_n = \phi_0^{n-1} \circ \phi_1, \  n \ge 1.
\end{equation}
We recall that the jump transformation $G$ of the Farey map is the well-known \emph{Gauss map}, related to the continued fraction expansion of real numbers.

\subsection{The symbolic representation} \label{sec:symb}
There is a natural way to encode the orbits of $F$ and its jump transformation $G$. Using the fact that $J$ forms a Markov partition of $[0,1]$ for $F$, given the set $\Omega = \{0,1\}^{\N_{0}}$ of infinite strings on two symbols, we define the coding $\pi : \Omega \to [0,1]$ as below. For $\omega = (\omega_{i})_{i \ge 0} \in \Omega$, define,
    \[
        \pi(\omega) := x \ \ \text{iff} \ \ F^{i}(x) \in J_{\omega_{i}} \ \ \forall j \ge 0.
    \]
Fixed $\theta \in (0,1)$, we endow $\Omega$ with the standard metric described as, 
    \[
        d_{\theta}(\omega,\omega') = \theta^{\inf \{ i \ge 0 : \omega_{i} \neq \omega'_{i} \}}, \quad  \omega, \omega' \in \Omega.
    \]
Then, $\pi$ is a homeomorphism on $[0,1] \setminus \cup_{i\ge 1} F^{-i}(1)$, and letting $T$ denote the left shift map on $\Omega$ defined by $(T\omega)_{i} = \omega_{i+1}$, $i\ge 0$, $\pi$ is a conjugacy, modulo a countable set, between the system $\left([0,1], F \right)$ and the well-known \emph{full shift on two symbols} $\left( \Omega , T \right)$. In fact, $\pi \circ T = F \circ \pi$. 

We call \emph{word} a finite string $u \in \cup_{l \ge 1} \{ 0,1\}^{l}$. In the following, we let $0^l$, $l\ge 1$, denote the word $00\dots 0$ of $l$ consecutive $0$'s, and $0^{\infty}$ denote the constant string $00\dots \in \Omega$. For any word $u$ of length $|u| = l \ge 1$,
\[
[u]:=\{ \omega \in \Omega : \omega_{0} = u_{0}, \dots, \omega_{l-1} = u_{l-1} \}
\]
denotes the cylinder in $\Omega$ of infinite strings containing $u$ has initial word. Note that 
\[
\pi([0^i]) = [0,a_i], \ \ i\ge 1, \quad \pi(0^\infty) = 0.
\]

Similarly to the construction of $G$, we consider the cylinder $[1]\subset \Omega$ as the reference set, and we look at the level sets $A_{n}$ of the function $\tau : \Omega \to \N$ defined as\footnote {We use the same notations in this setting as we did for $([0,1],F)$, as these are essentially the same sets under $\pi$.}
\[
         \tau(\omega) := 1 + \inf \{ i \ge 0 \, : \, \omega_{i} = 1 \} , \ \ \ \omega = \omega_{0} \omega_{1} \omega_{2} \dots .
\]
As before, $\tau - 1$ is the hitting time of $\omega$ to $[1]$, and $\tau$ is finite on $\Omega \setminus\{0^{\infty}\}$. For $n \ge 1$, the level set $A_{n} := \{ \omega \in \Omega : \tau(\omega) = n\}$  is the cylinder $[0^{n-1} 1]$ and $\{A_{n}\}_{n \ge 1}$ form a partition of $\Omega \setminus\{0^{\infty}\}$. Then $T^{\tau} : \Omega \setminus \{ 0^{\infty}\} \to \Omega$ is the jump transformation of $T$. 

For $k \ge 1$ and $\omega \in \Omega \setminus \{ 0^{\infty}\}$ define $\tau_{k} : \Omega \setminus \{ 0^{\infty}\} \to \N$ by,
\[
\tau_0(\omega) := \tau(\omega)\, ,\quad  \tau_{k}(\omega) := 1+ \inf \{ i \ge \tau_{k-1} : \omega_{i} = 1\}, \ \text{for } k \ge 1.
\]
Note that $(\tau_{k}-1)_{k \ge 0}$ is the sequence of successive hitting times of $\omega$ to $[1]$. For a given $\omega \in \Omega \setminus \{0^{\infty}\}$, $\tau_{k}(\omega) < \infty$ for all $k \ge 1$ if and only if $\omega \notin \bigcup_{i \ge 1} T^{-i}(0^{\infty}) $. Denote by 
\[
\Omega_{0}:= \Omega \setminus \bigcup\limits_{i \ge 0} T^{-i}(0^{\infty}).
\]
For $\omega \in \Omega_{0} $ and $k \ge 0$ define, $\sigma_{k} : \Omega \to \N$ by
    \[
        \sigma_{0} \equiv \tau_{0} \ \ \text{and} \ \  \sigma_{k} = \tau_{k} - \tau_{k-1}, \ \ \text{for} \ \ k \ge 1.
    \] 
which describes the sequence of times between passages into $[1]$. 

For $\Sigma := \N^{\N_0}$, endowed with the metric $d_\theta$, let $(\Sigma,S)$ denote the full shift on countably many symbols. Let 
\begin{equation} \label{iota}
       \iota : \Omega_{0} \longrightarrow \Sigma\, ,\quad  \iota(\omega):= \sigma =(\sigma_{0}(\omega) \sigma_{1}(\omega) \dots) \in \Sigma , \ \ \text{for} \ \ \ \omega \in \Omega_{0} .
\end{equation}
which is a homeomorphism and makes the following diagram commute.
    \[
        \begin{tikzcd}
         \Omega_{0} \arrow[r, "T^{\tau_{0}}"] \arrow[d, " \iota "]
        & \Omega_{0} \arrow[d, "\iota"] \\
        \Sigma \arrow[r, "S" ]
        & \Sigma
        \end{tikzcd}
     \]
Observe that the jump transformation $([0,1], G)$ is conjugate to the full shift on countable symbols $(\Sigma, S)$. In particular, $\iota([A_{n}]) = [n]$ for every $n \ge 1$.

\subsection{Transfer operators}
Finally, we recall the notion of \emph{transfer operator} for the maps $F$ and $G$, and their symbolic counterparts, $T$ and $S$. 

The transfer operator for $F$ is defined as the linear operator which sends a function $f$ on $[0,1]$ to the function
\[
(\CL f)(x) := \sum_{F(y)=x}\, \frac{1}{|F'(y)|}\, f(y)\, .
\]
This operator is well defined when the sum is finite or, in general, convergent. In our case, for each $x\in (0,1)$ there are two pre-images, $\phi_0(x)$ and $\phi_1(x)$. Hence,
\[
        (\CL f)(x) := |\phi'_{0}(x)|\, f(\phi_{0}(x)) + |\phi'_{1}(x)| \,f(\phi_{1}(x)).
\]
In the following, it will be useful to consider separately the contributions to $\CL f$ coming from the interval $J_0$ and $J_1$, hence we introduce the notation
\[
        (\CL_0 f)(x) := |\phi'_{0}(x)|\, f(\phi_{0}(x))\, ,\quad (\CL_1 f)(x) := |\phi'_{1}(x)| \,f(\phi_{1}(x)),
\]
so that $\CL f = \CL_0 f + \CL_1 f$.

For the jump transformation $G$, using \eqref{GAn} and \eqref{GAn-inv}, we define the the transfer operator-valued power series, denoted by $\CM_z$, as,
\[
        (\CM_z g) (x) = \sum\limits_{n\ge 1} z^n\, |\zeta'_{n}(x)| g(\zeta_{n}(x)),\quad z\in \C
\]
which, for $g\in C([0,1])$, is convergent when $|z|\le 1$. As above, we introduce the notation $\CM_n$ for the contribution to $\CM_z$ corresponding to the branch on $A_{n}$. That is,
\[
        (\CM_n g)(x) := |\zeta'_{n}(x)|\, g(\zeta_{n}(x))\, , \, n\ge 1,
\]
and $\CM_z f = \sum_{n\ge 1}\, z^n \CM_n f$.

The relation \eqref{GAn-inv} between the inverse branches of $F$ and $G$ implies
\[
z^n \CM_n = \CL_1 \circ \CL_0^{n-1}, \quad n\ge 1,
\]
from which the following identity follows
\begin{equation}\label{rel-to-FG}
(1-\CM_z)\, (1-z\CL_0) = 1-z\CL.
\end{equation}
This relation has already been used in \cite{BI} for the Farey and the Gauss maps.

Using the symbolic representations of $F$ and $G$, we introduce the transfer operators for the shifts $(\Omega,T)$ and $(\Sigma, S)$. For simplicity, we use the same notation $\CL$ and $\CM_z$.

On the shift space $(\Omega, T)$, we obtain the transfer operator $\CL$ defined as follows.
\[
        (\CL f)(\omega) := \exp ({V}(0\omega)) \,f(0\omega)+\exp (V(1\omega)) \, f(1\omega),
\]
where $V : \Omega \to \R$ denotes the potential,
\begin{equation}\label{pot-V}
        V(\omega) := \log (|\phi'_{\omega_{0}}(\pi(\omega_{1} \omega_{2} \dots))|),  \ \ \text{ for } \  \omega= \omega_{0} \omega_{1} \omega_{2} \dots .
\end{equation}
As above, $\CL_0$ and $\CL_1$ denote the two terms in $\CL$ corresponding to the two inverse branches of $T$ on $[0]$ and $[1]$, respectively.
 
For the jump transformation $T^{\tau_0}$, the potential $W:\Omega \setminus\{0^{\infty}\} \to \R$ to use is the \emph{induced potential of $V$} defined as, 
\begin{equation}\label{eq:jump_potential}
        W(\omega) := \sum\limits_{i=0}^{\tau_{0}(\omega)-1}V(T^{i}(\omega)).
\end{equation}
 Further, we need to consider the corresponding potential $\iota *W := W \circ \iota^{-1}$ on $\Sigma$. For simplicity of notation, we will denote this potential still as $W$. It will be clear from the context which potential is considered. 
 
 The transfer operator-valued power series for the shift map $S$ is then defined as
\[
        (\CM_z g) (\sigma) := \sum\limits_{i\ge 1} z^i\, \exp( {W}(i\sigma)) \, g(i\sigma) = \sum\limits_{i \ge 1} z^i\, (\CM_{i}g)(\sigma), \quad z\in \C.
\]
By use of the homeomorphism $\iota : \Omega_{0} \longrightarrow \Sigma$, we can lift $\CM_z$ to an operator-valued series acting on functions defined on $\Omega_0$. In the following, we use the same notation for the two operators, as the context will make it clear the underlying space being considered. In particular, we will make use of  relation \eqref{rel-to-FG} also in the symbolic setting.

The functional analytic properties of the transfer operators $\CL$ and $\CM_z$ depend on the properties of the maps $F$ and $G$. First, we note that assumptions (H2)-(H5) on $F$ imply the following properties for the potential $V$.
\begin{itemize}
\item[(V1)] There exists $0 < C_1 < \infty$ such that $-C_1 \le V(\omega) \le 0$ for all $\omega \in \Omega$, $V(0^{\infty}) = 0$.
\item[(V2)] The \emph{topological pressure} of $V$
\begin{equation} \label{pressure}
\CP(V) := \lim_{n\to \infty}\, \frac 1n\, \log\, \sum_{\omega_0,\omega_1,\dots, \omega_{n-1}}\, \exp \Big( \sup_{\omega'} \sum_{j=0}^{n-1}\, V(\omega_j\dots\omega_{n-1}\omega') \Big)
\end{equation}
satisfies $\CP(V)=0$.
\end{itemize}
In particular, (V2) holds because the transfer operator $\CL$ acting on $L^1(m)$, with $m$ the Lebesgue measure, has spectral radius 1.

In addition, we assume the properties below for the potential $W$.

\begin{itemize}
    \item[(W1)] $W$ is locally H\"older continuous on $\Sigma$, that is, there exists $0 < C_2 < \infty$ such that for all $n \ge 1$ and the constant $\theta \in (0,1)$ defining the metric $d_\theta$,
    \[ 
        var_{n}(W) := \sup \lbrace |W(\sigma) - W(\sigma')|: \inf \{ j \ge 0 \ : \ \sigma_{j} \neq \sigma'_{j} \} = n \rbrace \le C_2 \theta^{n}. 
    \]
   \item[(W2)] $W$ is {\it summable}, that is, there exists $0 < C_3 < \infty$ such that 
    \[ 
        \sum\limits_{n \in \Sigma} \exp \left( \sup\limits_{\sigma \in [n]} W(\sigma) \right) < C_{3}.
    \] 
\end{itemize}
 
Let $\CH_{\theta}(\Sigma)$ be the set of of complex valued bounded Lipschitz continuous (with respect to the metric $d_{\theta}$) functions on $\Sigma$. For $g\in \CH_{\theta}(\Sigma)$, let
\[
|g|_{\theta} := \sup \set{\frac{var_{n}g}{\theta^{n}} \ : \ n \ge 1},
\]
and set
\[
        \| g \|_{\theta} := \|g\|_{\infty} + |g|_{\theta},
\]
where $\|\cdot \|_{\infty}$ is the supremum norm. Then, $(\CH_{\theta}(\Sigma), \|  \cdot \|_{\theta})$ is a Banach space.

We note that the potential $\CW$ is not bounded below, and therefore $\CW \notin \CH_{\theta}(\Sigma)$. However, by the virtue of (W1)-(W2), the function $\exp \CW$ belongs to $ \CH_{\theta}$.

To conclude this section, we recall the spectral properties of the power series $\CM_z$ when acting on $\CH_{\theta}(\Sigma)$ that we need in the following. We use the notation $\CP(z)$ for the topological pressure, defined as in \eqref{pressure}, of the potential
\[
\CW_z(\sigma) := \CW(\sigma) + \sigma_0\, \log z\, .
\]

\begin{theorem}[Ruelle-Perron-Frobenius Theorem for $\CM_z$ \cite{Is}] \label{thm:RPF_M}
Assuming (V1)-(V2) and (W1)-(W2), for $|z|\le 1$ and $\CM_z$ acting on $\CH_{\theta}(\Sigma)$, the following hold:
\begin{itemize}
\item[(i)] The spectral radius of $\CM_z$ is bounded above by $\exp \CP(|z|)$.
\item[(ii)] There is at most one eigenvalue of modulus $\exp \CP(|z|)$ which is simple, and exactly one if $z$ is real and positive.
\item[(iii)] $\CP(1)=0$ and, if $z\not= 1$, then $1$ is not an eigenvalue of $\CM_z$.
\end{itemize}
Let us now consider the special case $z=1$ and let $\CM:= \CM_1$.Then, $1$ is the largest simple eigenvalue of $\CM$ with positive eigenfunction $h \in \CH_{\theta}(\Sigma)$ bounded from above and away from 0. The dual operator $\CM^*$ has a corresponding probability eigenmeasure $\nu \in \CH_{\theta}^{*}(\Sigma)$ with eigenvalue $1$. That is,
        \[
        \CM h = h \ \ \text{and} \ \ \CM^{*} \nu= \nu.
        \]
The measure $\rho := h\nu$ is the $S$-invariant Gibbs measure corresponding to the potential $\CW$ on $\Sigma$, that is, there exists a constant $0<C_4<\infty$ such that, for all $k\ge 1$ and $\sigma_0,\dots,\sigma_{k-1} \in \Sigma$,
\[
\frac{1}{C_4} \le \frac{\rho([\sigma_0 \sigma_1 \dots \sigma_{k-1}])}{\exp \sum_{j=0}^{k-1}\, \CW(S^j \sigma)} \le C_4\, ,
\]
for all $\sigma \in [\sigma_0 \sigma_1 \dots \sigma_{k-1}]$ (recalling that $\CP(\CW)=\CP(1)=0$).\\ 
In addition, the function $e:= \sum_{j\ge 0} \CL_0^j h$ satisfies $\CL e =e $, and the measure $\mu:= e \iota^*(\nu)$ on $\Omega_0$ is $T$-invariant\footnote{Here, $\iota^*(\nu)$ denotes the pull-back of the measure $\nu$ to $\Omega_0$ by \eqref{iota}. Moreover, the measure $\mu$ coincides, up to the homeomorphism $\pi$, with the invariant measure in (H5).}. 
\end{theorem} 
 
\section{The open systems and the escape rate} \label{sec:open}

For the parabolic maps $F$ of the interval defined in the previous section, we now introduce the \emph{Markov holes} containing the indifferent fixed point. Recalling the sequence $\{a_n\}_{n\ge 0}$ defined in \eqref{preimages}, let $H:=[0,a_{N}]$ for fixed $N\ge 2$, which can be written as $H =\bigcup_{n\ge N+1}A_{n}$, a collection of contiguous elements of the partition $\{A_{n}\}_{n \ge 1}$. Then the corresponding hole in the symbolic setting $(\Omega,T)$ is simply the union of cylinders of the form $H= \bigcup_{i \ge N} [0^{i} 1]$ in $\Omega$. Denote by $\Omega_{H} := \Omega \setminus \bigcup_{i \ge 0} T^{-i}(H)$, the survivor set. 

In the induced symbolic system $(\Sigma, S)$, the hole becomes 
\[
\hat{H}:=\iota(H) = \{ N+1, N+2, \dots \} =\bigcup\limits_{k \ge N+1}[k] \subset \Sigma,
\] 
and the survivor set is denoted by $\Sigma_{N} := \{ \sigma \in \Sigma : \sigma_{i} \le N \ \forall i \ge 0\}$. The set $\Sigma_{N}$ is closed under the action of $S$, and the system $(\Sigma_{N},S)$ is a full shift over the compact finite symbol set $\{ 1, 2, \dots, N \}$. The survivor sets are topologically conjugate since the following diagram commutes
    \[
    \begin{tikzcd}
     \Omega_{H} \arrow[r, "T^{\tau_{0}}"] \arrow[d, " \iota "]
    & \Omega_{H} \arrow[d, "\iota"] \\
    \Sigma_{N} \arrow[r, "S" ]
    & \Sigma_{N}
    \end{tikzcd}
     \]
    
Our aim is to compute the \emph{escape rate} $\gamma_\mu(H)$ of the hole $H$ with respect to the measure $\mu$ of assumption (H5), that is
\begin{equation}\label{escape-rate-F}
\gamma_\mu(H) := \limsup_{n\to \infty}\, \frac {-\log \Big( \mu\set{x\in [0,1]\, :\, F^j(x) \not\in H\, ,\, \forall\, j=0,\dots, n-1} \Big)}{n}\, .
\end{equation}
By the assumptions on the maps $F$, the measure $\mu$ has bounded density on $[0,1]\setminus H$, therefore the escape rate $\gamma_\mu(H)$ coincides with $\gamma_m(H)$, with $m$ the Lebesgue measure. Our main result shows the asymptotic behavior of $\gamma_\mu(H)$ as $H$ shrinks around 0, that is as $N\to \infty$.

First, the homeomorphism $\pi$ implies that the escape rate $\gamma_\mu(H)$ is the same in the symbolic setting. That is, we consider in \eqref{escape-rate-F} the measure $\mu$ on $\Omega_0$ of Theorem \ref{thm:RPF_M} and the set of strings $\omega\in \Omega_0$ for which $T^j(\omega) \not\in \cup_{i \ge N} [0^{i} 1]$ for all $j=0,\dots,n-1$.
    
Then, following \cite{DY} and \cite[Section 7.1]{DT}, the escape rate $\gamma_\mu(H)$ may be computed by using the transfer operator of the open system. Therefore, we introduce the transfer operator $\CQ : \CH_{\theta}(\Omega) \to \CH_{\theta}(\Omega)$ and the transfer operator-valued series $\CN_z:\CH_{\theta}(\Sigma) \to \CH_{\theta}(\Sigma)$ for the open systems in $\Omega$ and $\Sigma$ respectively, given by,
\[
\begin{aligned}
\CQ f := &\, \CL ((1-\chi_{H})f) = \CL_{0} ((1-\chi_{H})f) + \CL_{1}f, \\[0.2cm]
\CN_z g = &\,  \CM_z((1-\chi_{\hat{H}})g) = \sum\limits_{i=1}^{N} z^i\, \CM_i g = \sum\limits_{i=1}^{N} z^i\, (\CQ_1 \circ \CQ_0^{i-1}) g, \quad z\in \C,
\end{aligned}
\]
where $\CQ_{0}f:= \CL_{0} ((1-\chi_{H})f)$, $\CQ_{1}:= \CL_{1}$. In the following we denote $\CN_i := \CM_i$ for $1 \le i \le N$. Standard computations show that \eqref{rel-to-FG} holds also in the open setting, hence 
\begin{equation}\label{eq:relation_N_Q}
    (1-\CN_z)(1-z\CQ_{0}) = 1-z\CQ.
\end{equation}

Outside of the hole $H$, the maps $F$ are expanding, hence we expect the transfer operator $\CQ$ on $\CH_{\theta}(\Omega)$ to have a maximal real eigenvalue $\lambda_{\CQ} = \exp \CP(V|_{\Omega_H})$, where $V|_{\Omega_H}$ is the potential \eqref{pot-V} restricted to the survivor set $\Omega_H$ and $\CP$ denotes the pressure \eqref{pressure}, and a spectral gap. Therefore, by \cite{DY} and \cite[Section 7.1]{DT}, we have
\begin{equation}\label{escape-rate-T-to}
\gamma_\mu(H) = \CP(V) - \CP(V|_{\Omega_H}) = -\log \lambda_{\CQ}\, ,
\end{equation}
where we use that $\CP(V)=0$.

The same relation holds for the jump transformation. We know due to \cite{Tanaka}, that $\CN:= \CN_1$ has a largest simple real eigenvalue $\lambda_{N} < 1$ which also satisfies $ \lambda_{N} = \exp \CP(\CW|_{\Sigma_{N}})$, where $\CW|_{\Sigma_{N}}$ is the potential \eqref{eq:jump_potential} projected on $\Sigma$ and restricted to the survivor set $\Sigma_N$. Further, the escape rate into $\hat{H}$ with respect to the measure $\rho$ of Theorem \ref{thm:RPF_M}, is the difference of topological pressures as below,
\begin{equation}\label{escape-rate-S-to} 
        \gamma_\rho(\hat{H}) = \CP(\CW) - \CP(\CW|_{\Sigma_{N}}) = -\log \lambda_{N}. 
\end{equation}
where again $\CP(\CW)=0$.

Making use of the fact that the survivor set $\Sigma_{N}$ is the full shift over finitely many symbols, we consider the transfer operator $\CN$ acting on the function spaces on the survivor set, that is, $\CN:\CH_{\theta}(\Sigma_{N}) \to \CH_{\theta}(\Sigma_{N})$. Here the classical Ruelle-Perron-Frobenius theorem holds (see e.g. \cite{PP}) and $\lambda_{N}$ is the leading simple eigenvalue of $\CN$ acting on $\CH_{\theta}(\Sigma_{N})$. Let $h_{N} \in \CH_{\theta}(\Sigma_{N})$ and $ \nu_{N} \in \CH_{\theta}^{*}(\Sigma_{N})$ denote the corresponding eigenfunction and eigenmeasure of $\CN$ respectively, that is,
\begin{equation}\label{rpf-sigmaN} 
        \CN h_{N} = \lambda_{N}h_{N} \ \ \text{and} \ \ \CN^{*}\nu_{N} = \lambda_{N}\nu_{N}.
 \end{equation}
Recall that $\nu_{N} $ is an $S$-invariant probability measure on $\Sigma_{N}$ and the measure $\rho_{N} := h_{N} \nu_{N}$ is the Gibbs measure for $\CN$ for the potential $\CW|_{\Sigma_{N}}$ on $\Sigma_{N}$. 

In the next results, we prove the relation between $\gamma_\mu(H)$ and $\gamma_\rho(\hat{H})$ which connects the original system with its accelerated version.

\begin{theorem} \label{thm:escape_H_hatH}
    Consider the parabolic dynamical system on the interval satisfying (H1)-(H5) and (W1)-(W2), and fix a Markov hole $H:=[0,a_N]$ for $N \ge 2$, with $a_N$ defined as in \eqref{preimages}. Then, use the symbolic representation introduced in Section \ref{sec:symb} and above for the open system. Let $h_{N}, \nu_{N}$ be the eigenfunction and eigenmeasure of $\CN$ on $\CH_{\theta}(\Sigma_{N})$ as in \eqref{rpf-sigmaN}, and $\rho_{N} := h_{N} \nu_{N}$. Define $e_N := \sum_{k=0}^{N-1} \CQ_{0}^{k}h_{N}$. Then,
     \begin{itemize}
         \item[(i)] The measure $\mu_{N} := e_N \iota^*(\nu_{N})$ is a finite $T$-invariant measure supported on $\Omega_{H}$.
         \item[(ii)] The corresponding probability measure $\hat{\mu}_{N} := \mu_{N}/\mu_{N}(\Omega)$ is an equilibrium state for the potential $V|_{\Omega_{H}}$ on the survivor set. In particular, the escape rates into the hole $H$ in $\Omega$ and $\hat{H}$ in $\Sigma$ satisfy 
         \begin{equation} \label{aim-1}
             \gamma_\mu(H) = \frac{\gamma_\rho(\hat{H})}{\sum\limits_{k=1}^{N} k \cdot \, \rho_{N}([k])}
         \end{equation}
     \end{itemize}
\end{theorem}

\begin{proof} 
In order to simplify the notation, we identify the subsets of $\Omega_0$ and of $\Sigma$ which are related by the homeomorphism $\iota$ defined in \eqref{iota}. It will be clear from the operators and the measures that we use, whether we are referring to a subset $E\subset \Omega_0$ or to $\iota(E) \subset \Sigma$.

Recall the sets $A_{k} = [0^{k-1}1]$ in $\Omega$, for which $\iota(A_k)=[k]$, and construct the sets 
\[
D_{i} :=\bigcup\limits_{j = i+1}^{N}A_{j}|_{\Omega_{H}} = \bigcup\limits_{j \ge i}^{N-1}[0^{j}1]|_{\Omega_{H}},
\]
for $0 \le i \le N-1$. We omit the restriction of these sets to $\Omega_{H}$ from the notations whenever the context is clear.

We claim that for any Borel set $E$ of $\Omega_{H}$ we have
\begin{equation}\label{claim-1}
\mu_{N}(E) = \sum\limits_{k =0}^{N-1} \rho_{N}(T^{-k}E \cap D_{k})\, .
\end{equation}

The proof of \eqref{claim-1} goes as follows. Since each $D_{k}$ is a disjoint union of the sets $A_{k+1}|_{\Omega_{H}}, \dots, A_{N}|_{\Omega_{H}},$ we can write $\rho_{N}(T^{-k}E \cap D_{k}) = \sum_{j=k+1}^{N}\rho_{N}(T^{-k}E \cap A_{j})$ and observe that 
\[
\rho_{N}(T^{-k}E \cap A_{j}) = \int_{\Omega_{H}} (\chi_{E}\circ T^{k}) \chi_{[0^{j-1}1]}h_{N} d\nu_{N}\, .
\]
Using the fact that $\CN^{*}\nu_{N} = \lambda_{N} \nu_{N}$, and $\CN_{\ell}$ corresponds to that term where preimage of a point lies in the set $A_{\ell}$, we get, setting $\psi_\CW := \exp \CW$,
\begin{eqnarray}
    \rho_{N}(T^{-k}E \cap A_{j}) & = & \frac{1}{\lambda_{N}}\int_{\Omega_{H}} \sum\limits_{\ell=1}^{N} \CN_\ell \left((\chi_{E}\circ T^{k})\, \chi_{[0^{j-1}1]}\,h_{N}\right) d\nu_{N} \nonumber \\
    &=& \frac{1}{\lambda_{N}}\int_{A_{j}} (\chi_{E}\circ T^{k}) \, \psi_{\CW} \, h_{N} \, d\nu_{N} \nonumber \\
   \implies \rho_{N}(T^{-k}E \cap D_{k}) & = & \frac{1}{\lambda_{N}}\int_{D_{k}} (\chi_{E}\circ T^{k}) \, \psi_{\CW} \, h_{N} \, d\nu_{N} \label{eq:rho_N}
\end{eqnarray}
On the other hand, 
\[
\mu_{N}(E) = \sum\limits_{j=1}^{N}\mu_{N}(E \cap A_{j}) = \sum\limits_{j=1}^{N} \int_{\Omega_{H}} \chi_{E} \, \chi_{A_{j}} \left( \sum\limits_{k=0}^{N-1}\CQ_{0}^{k}h_{N}\right) \, d\nu_{N}\, .
\]
 Again using the fact $\CN^{*}\nu_{N} = \lambda_{N} \nu_{N}$ and appropriate change of variables we get,
    \[ 
        \mu_{N}(E) = \frac{1}{\lambda_{N}} \sum\limits_{j=1}^{N} \sum\limits_{k=0}^{N-j} \int_{A_{k+j}} (\chi_{E} \circ T^{k}) \, \psi_{\CW} \, h_{N} \, d\nu_{N} = \frac{1}{\lambda_{N}} \sum\limits_{k=0}^{N-1} \sum\limits_{j=1}^{N-k} \int_{A_{k+j}} (\chi_{E} \circ T^{k}) \, \psi_{\CW} \, h_{N} \, d\nu_{N}. 
 \]
By \eqref{eq:rho_N},
\[
\frac{1}{\lambda_{N}} \sum\limits_{j=1}^{N-k} \int_{A_{k+j}} (\chi_{E} \circ T^{k}) \, \psi_{\CW} \, h_{N} \, d\nu_{N}=\rho_{N}(T^{-k}E \cap D_{k})\, ,
\]
hence \eqref{claim-1} is proved.

A first consequence of \eqref{claim-1} is that the measure $\mu_{N}$ is finite, since
\begin{equation} \label{misura}
\mu_{N}(\Omega_{H}) = \sum\limits_{k =0}^{N-1} \rho_{N}( D_{k}) = \sum\limits_{k =1}^{N} \sum\limits_{j =k}^{N}\rho_{N}( A_{j}) = \sum\limits_{k=1}^{N} k \, \rho_{N}(A_{k})\, .
\end{equation}
Finally, arguing as in the proof of Theorem \ref{thm:RPF_M}, using \eqref{eq:relation_N_Q}, we establish that $\mu_{N}$ is a $T$-invariant measure supported on $\Omega_{H}$ because $\rho_{N}$ supported on the survivor set in $\Sigma$ is $S$-invariant. This ends the proof of (i).

To prove (ii), we adapt arguments from \cite{Is}. First, we recall Abramov's formula, which relates the metric entropies of a system and of its jump transformation. In particular, if $h_{\hat{\mu}_{N}}(T)$ and $h_{\rho_{N}}(S)$ denote the metric entropies of $(\Omega_H,T)$ with respect to $\hat{\mu}_{N}$ and of $(\Sigma_N,S)$ with respect to $\rho_{N}$, respectively, then
\[
\frac{1}{\mu_{N}(\Omega_{H})}\, h_{\rho_{N}}(S) = h_{\hat{\mu}_{N}}(T)\, .
\]
Furthermore, we observe that $\int_{\Omega_{H}}V \, d\mu_{N} = \int_{\Omega_{H}}\CW \, d\rho_{N}$. This follows from
\[
\begin{aligned}
\int_{\Omega_{H}}V \, d\mu_{N} = &\, \int_{\Omega_{H}} V \cdot \Big( \sum_{k=0}^{N-1}\, \CQ_0^k h_N \Big)\, d\nu_N = \sum_{k=0}^{N-1}\, \int_{D_k}\, (V\circ T^k)\, h_N\, d\nu_N \\[0.2cm]
= &\, \sum_{i=1}^N\, \int_{A_i\cap \Omega_H}\, \Big( \sum_{j=0}^{i-1}\, V\circ T^j \Big)\, h_N\, d\nu_N = \int_{\Omega_{H}}\CW \, d\rho_{N}\, ,
\end{aligned}
\]
where we have used the definition of the potential $V$ \eqref{pot-V}, the definition of the sets $D_i$, and \eqref{eq:jump_potential}.

Finally, we use the variational principle for $(\Sigma_N,S)$ with potential $\CW|_{\Sigma_{N}}$, to obtain that the topological pressures of $\CW|_{\Sigma_{N}}$ and $V|_{\Omega_{H}}$ are related as,
\begin{eqnarray*}
    \frac{\CP(W|_{\Sigma_{N}})}{\mu_{N}(\Omega_{H})}  =  \frac{1}{\mu_{N}(\Omega_{H})} \left( h_{\rho_{N}}(S) + \int_{\Omega_{H}}\CW \, d\rho_{N} \right) & = &  h_{\hat{\mu}_{N}}(T) + \int_{\Omega_{H}}V \, d\hat{\mu}_{N}\, .\end{eqnarray*}
Finally, since $\rho_{N}$ is an equilibrium state for $\CW|_{\Sigma_{N}}$, the same holds for $\hat{\mu}_{N}$ with respect to $V|_{\Omega_{H}}$. Hence,
\[
\frac{\CP(W|_{\Sigma_{N}})}{\mu_{N}(\Omega_{H})}  = \CP(V|_{\Omega_{H}})\, .
\]
The proof is finished by using \eqref{escape-rate-T-to}, \eqref{escape-rate-S-to}, and \eqref{misura}.
\end{proof}

\begin{rem}
Relation \eqref{aim-1} can be written in terms of \eqref{eq:relation_N_Q} for the transfer operators. It says that the smallest real $\bar z$ for which the operator $\CN_z$ has eigenvalue 1 is given by 
\[
\bar z = \lambda_\CQ^{-1} = \exp (\gamma_\mu(H)) = \exp\Big(\gamma_\rho(\hat{H})\Big/ \sum\limits_{k=1}^{N} k \cdot \, \rho_{N}([k])\Big)\, .
\]
\end{rem}

We are now ready to study the asymptotic behaviour of $\gamma_\mu(H)$ as $H$ shrinks around the indifferent fixed points, that is as $N\to \infty$ and $a_N\to 0^+$. First, in the next example, we show how our approach works in the simple case of the piecewise linear maps $F$ studied in \cite{MK}, obtaining the same results. Then, we prove that our approach works also in the general case of a differential parabolic map.

\begin{example} \label{ex:pl}
The simplest case of a parabolic map on the interval $[0,1]$ is that of a piecewise linear map. Given a countable partition $\{A_k\}_{k\ge 1}$ of $[0,1]$ in intervals $A_k=(a_k,a_{k-1}]$ with $a_0:=1$ and $\{a_k\}$ a strictly decreasing sequence converging to $0$, let $J_0 =   \cup_{k\ge 2} A_k \cup \{0\}= [0,a_1],$ and $J_1= A_1$. Then, we consider a map $F$ which is linear when restricted to any interval $A_k$, continuous on $J_0$ and $J_1$, and strictly increasing on $J_0$. Such a map may satisfy (H1)-(H5), except for (H2), and (H3) is obtained by letting $a_k \sim k^{-\frac 1s}$ as $k\to \infty$.

Let $p_0:=1$, and $p_k := a_{k-1}-a_k$ for $k\ge 1$, then $\sum_k p_k=1$. In the symbolic representation $(\Omega,T)$, the potential $V$ of \eqref{pot-V} reads in this case
\[
V(\omega) = \begin{cases}
    - \log \frac{p_{k-1}}{p_k}, &\text{if $\tau(\omega)=k$},\\
    0, &\text{if }\ \omega = 0^{\infty}.\\
\end{cases}
\]
In the same way, the induced potential \eqref{eq:jump_potential} is given by
\[
W(\omega) = \log p_k,\quad \text{if $\tau(\omega)=k$.}
\]
Finally, in the accelerated symbolic system $(\Sigma,S)$, we obtain the potential 
\[
W(\sigma) = \log p_{\sigma_0}, \quad \text{for all $\sigma \in \Sigma.$}
\]
Thus we obtain a one-body potential $W$ on $\Sigma$, and Theorem \ref{thm:RPF_M} gives the $S$-invariant Gibbs product measure $\rho$ with $\rho([k]) = p_k$ for all $k\ge 1$.

Let's now consider the open systems with hole $H=[0,a_N]$ or $\hat{H} = \cup_{k\ge N+1} [k]$. This case is particularly easy to study since the map on the survivor set $\Sigma_N$ reduces to a full shift on $\{1,\dots,N\}^{\N_0}$ with the one-body potential $W|_{\Sigma_N}$. Hence, it is convenient to let the transfer operator $\CN$ act on $\C^N$, the set of functions $g$ which satisfy $g(\sigma) = g(\sigma_0)$ for all $\sigma\in \Sigma_N$. Then, it is a standard result that $\CN$ reduces to the action of the $N\times N$ matrix 
\[
L_{\CN}= \begin{pmatrix}
    p_{1} & p_{2} &\cdots &p_{N}\\
    p_{1} & p_{2} &\cdots &p_{N}\\
          &       & \vdots &  \\
    p_{1} & p_{2} &\cdots &p_{N}
\end{pmatrix}
\]
The Perron root $\lambda_{N}$ of $L_{\CN}$, its largest simple real eigenvalue, and the corresponding eigenfunction $h_{N}$ satisfy,
\[ 
\lambda_{N} = \sum_{i=1}^{N} p_{i} \quad \text{and} \quad h_{N}(i) = h_{N}(j)\, \, \, \forall\, i,j \in \{1,\dots,N\}.
\]
In particular, the Gibbs measure $\rho_N$ used in Theorem \ref{thm:escape_H_hatH} is the product measure given by 
\[
\rho_N([k]) = \frac{p_k}{\sum_{i=1}^{N} p_i}\, .
\]
Thus, by \eqref{aim-1}, the escape rate $\gamma_\mu(H)$ of the parabolic system satisfies
\[
\gamma_\mu(H) = \frac{\gamma_\rho(\hat{H})}{\sum\limits_{k=1}^{N} k \cdot \, \rho_{N}([k])} = \frac{ -\log \sum_{i=1}^{N} p_i}{\sum\limits_{k=1}^N\, \frac{kp_k}{\sum_{i=1}^{N} p_i}} = \frac{ -\log \sum_{i=1}^{N} p_i}{\sum\limits_{k=1}^N\, k\, p_k}\, \sum_{i=1}^{N} p_i \, .
\]
Finally, as $N\to \infty$, we have 
\[
-\log \sum_{i=1}^{N} p_i = -\log \Big(1-\sum_{k\ge N+1} p_k\Big) \sim \sum_{k\ge N+1} p_k = a_N = m(H)\, ,
\]
\[
\sum_{i=1}^{N} p_i \to 1\, .
\]
Let $m$ be the Lebesgue measure. Then, as $N\to \infty$, with $a_N \sim N^{-\frac 1s}$, we obtain for $s\not= 1$
\[
p_N \sim N^{-\frac 1s-1} \quad \text{and} \quad \sum\limits_{k=1}^N\, k\, p_k  \left\{ \begin{array}{ll}
= O(1)\, , & \text{for $s\in (0,1)$;}\\[0.2cm] \sim N^{-\frac 1s+1} \sim m(H)^{1-s}\, , & \text{for $s>1$;} \end{array} \right.
\]
and for $s=1$
\[
p_N \sim N^{-2} \quad \text{and} \quad \sum\limits_{k=1}^N\, k\, p_k \sim \log N \sim -\log m(H)\, .
\]
Thus, we obtain
\[
\gamma_\mu(H) \sim \left\{ \begin{array}{ll} m(H)\, , & \text{if $s\in (0,1)$;}\\[0.2cm] \frac{m(H)}{-\log m(H)}\, , & \text{if $s=1$;}\\[0.2cm] m(H)^s\, , & \text{if $s>1$.} \end{array} \right.
\]
These results are in accordance with those in \cite{MK}.
\end{example}

First, we make a couple of observations regarding the measures $\nu_{N}$ and $\rho_N=h_N\nu_N$ of \eqref{rpf-sigmaN}, and $\nu$ and $\rho=h\nu$ of Theorem \ref{thm:RPF_M}. We recall due to \cite{Is} that, as $N \to \infty$, the eigenvalue $\lambda_{N} \to 1$ and the sequences of eigenfunctions $h_{N}$ and eigenmeasures $\nu_{N}$ converge uniformly to $h$ and $\nu$ respectively. Moreover $h \asymp 1$.

\begin{lemma} \label{lem-utile}
Let $N \ge 2$ be fixed. Then\footnote{For two sequences $(c_{k})_{k \in \N}$ and $(d_{k})_{k \in \N}$, by $c_{k} \asymp d_{k}$ we mean that there exists $M > 0$ such that $M^{-1} \le c_{k}/d_{k} \le M$ for all $k \in \N$.}
\begin{itemize}
        \item[(i)] $\rho(A_{k}) \asymp \nu(A_{k})$. 
        \item[(ii)] $\rho(A_{k}) \asymp \psi_{W}(0^{k-1}1\omega) $ uniformly in $\omega \in \Omega$, with $\psi_W:= \exp\circ W$.
        \item[(iii)] For $\epsilon > 0$, there exists $N_{\epsilon}$ such that for all $N \ge N_{\epsilon}$ and for all $1 \le k \le N$,
        \[
            \big|\nu_{N}(A_{k}) - \nu(A_{k}) \big| < \epsilon \ \nu(A_{k}).
        \]
\end{itemize}
\end{lemma}

\begin{proof}
    The first and the second statements are already proved in \cite{Is}, we recall the proofs here for completion.\\
    For (i), $\rho(A_{k}) \asymp \nu(A_{k})$ can be proved using the fact that $h \asymp 1$ in $\rho(A_{k}) = \int_{A_{k}} h \, d\nu$. \\
    The statement (ii) can be proved by observing, 
    \[
    \nu(A_{k}) = \int_{\Omega} \chi_{A_{k}} d\nu =  \int_{\Omega} (\CM\chi_{A_{k}}) d\nu = \int_{\Omega} \psi_{W}(0^{k-1}1\omega) d\nu(\omega) = \psi_{W}(0^{k-1}1\omega^{*})
    \]
for some $\omega^{*} \in \Omega$, and by using (W1) which gives the inequality  
\[
e^{-C_{2} \theta}\le e^{-C_{2} \theta^k} \le \frac{\psi_{W}(0^{k-1}1\omega^{*})}{\psi_{W}(0^{k-1}1\omega)} \le e^{C_{2} \theta^k} \le e^{C_{2} \theta}
\] 
for all $\omega \in \Omega$.  \\
    For (iii), for any fixed $N$ and for all $1 \le k \le N$, consider,
    \[
        \big|\nu_{N}(A_{k}) - \nu(A_{k}) \big| = \big|\int_{\Omega} \chi_{A_{k}} d\nu_{N} - \int_{\Omega} \chi_{A_{k}} d\nu \big|.
    \] 
Using the fact that $\nu_{N}$ and $\nu$ are eigenmeasures for $\CN$ and $\CM$ respectively,
    \begin{eqnarray*}
         \big|\nu_{N}(A_{k}) - \nu(A_{k}) \big| & = & \Bigg| \frac{1}{\lambda_{N}} \int_{\Omega} (\CN \chi_{A_{k}}) (\omega) d\nu_{N}(\omega) - \int_{\Omega} (\CM \chi_{A_{k}}) (\omega) d\nu(\omega) \Bigg| \\
         & = & \Bigg| \frac{1}{\lambda_{N}} \int_{\Omega} \psi_{W}(0^{k-1}1\omega) d\nu_{N}(\omega) - \int_{\Omega} \psi_{W}(0^{k-1}1\omega) d\nu(\omega) \Bigg|\\
         & \le & \Bigg| \langle \frac{1}{\lambda_{N}}\nu_{N} - \nu, \  \| \psi_{W}(0^{k-1}1\, \cdot) \|_{\infty} \rangle  \Bigg|\\
         & \le & M \,\rho(A_{k}) \, \Big|\frac{1}{\lambda_{N}}-1\Big|,
    \end{eqnarray*}
    where $M$ is such that, by (ii), $\psi_{W}(0^{k-1}1\omega) \le M \rho(A_{k})$ for all $\omega \in \Omega$ and $k \ge 1$. Now using $\lambda_{N} \to 1$, for given $\epsilon>0$, choose $N_{\epsilon} \in \N$ such that $|1/\lambda_{N}-1| < \epsilon$ for all $N \ge N_{\epsilon}$, which proves the result.
\end{proof}

\begin{proposition} \label{prop-utile}
As $N\to \infty$, we have\footnote{For two sequences $c_{n}$ and $d_{n}$, by $c_{n} \sim d_{n}$ we mean $\lim_{n \to \infty} c_{n}/d_{n} = 1$.}
    \[ 
       \sum\limits_{k=1}^{N} k \cdot \, \rho_{N}(A_{k})  \sim \sum\limits_{k=1}^{N} k \cdot \, \rho(A_{k}). 
    \]
\end{proposition}

\begin{proof}
    For any finite $N$, 
    \[
    \left| \frac{\sum\limits_{k=1}^{N} k \cdot \, \rho_{N}(A_{k})}{\sum\limits_{k=1}^{N} k \cdot \, \rho(A_{k})} -1 \right| \le \frac{\sum\limits_{k=1}^{N} k \cdot \, \Big|\rho_{N}(A_{k}) - \rho(A_{k}) \Big|}{ \sum\limits_{k=1}^{N} k \cdot \, \rho(A_{k})}\, .
    \]
    Let $\epsilon> 0$ be fixed. For every $1 \le k \le N$,
    \begin{eqnarray*}
        \Big|\rho_{N}(A_{k}) - \rho(A_{k}) \Big| \le \| h_{N} \|_{\infty} \big|\nu_{N}(A_{k}) - \nu(A_{k}) \big| +  \| h_{N} - h \|_{\infty} \nu(A_{k}).
    \end{eqnarray*}
    Since $h_{N}$ converge to $h$ uniformly and $h \asymp 1$, there exists $M > 0$ and $N_{1} \in \N$ such that for all $N \ge N_{1}$, $\|h_{N}\|_{\infty} < M$. Moreover, there exists $N_{2} \in \N$ such that $\|h_{N} - h\|_{\infty} < \epsilon$ for all $N \ge N_{2}$. Now using Lemma \ref{lem-utile}, there exists $M' > 0$ such that for all $N \ge \max\{N_{\epsilon},\,  N_{1}, \, N_{2}  \}$
    \[
        \Big|\rho_{N}(A_{k}) - \rho(A_{k}) \Big| \le  \epsilon \,   M' \rho(A_{k}),
    \]
which implies,
    \[
         \left| \frac{\sum\limits_{k=1}^{N} k \cdot \, \rho_{N}(A_{k})}{\sum\limits_{k=1}^{N} k \cdot \, \rho(A_{k})} -1 \right| \le \epsilon \,   M' ,
    \]
proving the result.
\end{proof}

We can now prove

\begin{theorem}\label{thm:shrinking}
Consider the parabolic dynamical system on the interval satisfying (H1)-(H5) and (W1)-(W2), and fix a Markov hole $H:=[0,a_N]$ for $N \ge 2$, with $a_N$ defined as in \eqref{preimages}. Then, as $N\to \infty$, or equivalently as $m(H)\to 0^+$ with $m$ the Lebesgue measure, we have
\begin{equation}\label{final-formula}
\gamma_\mu(H) \sim \left\{ \begin{array}{ll} c_s \cdot m(H)\, , & \text{if $s\in (0,1)$;}\\[0.2cm] c_s \cdot \frac{m(H)}{-\log m(H)}\, , & \text{if $s=1$;}\\[0.2cm] c_s \cdot m(H)^s\, , & \text{if $s>1$,} \end{array} \right.
\end{equation}
where $c_s>0$ denotes a constant only depending on the system and not on the hole.
\end{theorem}

\begin{proof}
In the proof, $const$ denotes possibly different constants only depending on the system and not on the hole. By \cite{KL} we have that $\gamma_\rho(\hat{H}) \sim\, const\cdot m(H)$.

In addition, arguing as in \cite[Lemma 10.4]{Is}, we have
\begin{equation}\label{asymp-an}
m(H) = a_N \sim const \cdot N^{-\frac 1s}\, ,\qquad \rho(A_N) \sim const \cdot N^{-1-\frac 1s}
\end{equation}
as $N\to \infty$. Hence, using Proposition \ref{prop-utile} in \eqref{aim-1}, one finds that there exists a constant $c_s > 0$ such that,
\[
\gamma_\mu(H) \sim \frac{\gamma_\rho(\hat{H})}{\sum\limits_{k=1}^{N} k \cdot \, \rho(A_k)} \sim const\cdot \frac{m(H)}{N^{1-\frac 1s}} \sim const\cdot \frac{m(H)}{m(H)^{1-s}} \sim c_s\cdot m(H)^s\, .
\]
for $s>1$. Analogously,
\[
\sum\limits_{k=1}^{N} k \cdot \, \rho(A_k) \sim \sum\limits_{k=1}^{N} k^{-1} \sim \log N \sim -\log m(H)\, .
\]
implies the result for $s=1$.

The case $s\in (0,1)$ follows from \eqref{asymp-an}, since the denominator in $\gamma_\mu(H)$ is summable in this case.
\end{proof}

Finally, we discuss the case of a more general hole $H_\eps=[0,\eps]$ with $\eps>0$. 

\begin{corollary} \label{general-hole}
Consider the parabolic dynamical system on the interval satisfying (H1)-(H5) and (W1)-(W2), and fix a hole $H_\eps=[0,\eps]$ with $\eps>0$. Then, as $\eps\to 0^+$,
\[
\gamma_\mu(H_\eps) \sim \left\{ \begin{array}{ll} c_s \cdot m(H_\eps)\, , & \text{if $s\in (0,1)$;}\\[0.2cm] c_s \cdot \frac{m(H_\eps)}{-\log m(H_\eps)}\, , & \text{if $s=1$;}\\[0.2cm] c_s \cdot m(H_\eps)^s\, , & \text{if $s>1$,} \end{array} \right.
\]
where $c_s$ is a constant as in Theorem \ref{thm:shrinking}.
\end{corollary}

\begin{proof}
The proof relies on the observation that the asymptotic behaviour of escape rate for general holes $H_\eps$ is controlled by the asymptotic behaviour of escape rate of Markov holes as we describe below. It is immediate that for all $\eps>0$ there exists $N_\eps \in \N$ such that $a_{N_\eps+1} < \eps \le a_{N_\eps}$. Hence, we find two Markov holes, $H_{N_\eps+1} = [0,a_{N_\eps+1}]$ and $H_{N_\eps} = [0,a_{N_\eps}]$, such that $H_{N_\eps+1} \subset H_\eps \subseteq H_{N_\eps}$. 

The definition of the escape rate \eqref{escape-rate-F} implies that $\gamma_\mu(H_{N_\eps+1})\le \gamma_\mu(H_\eps)\le \gamma_\mu(H_{N_\eps})$. Hence, we can use \eqref{final-formula} to obtain a result about the behaviour of $\gamma_\mu(H_\eps)$ as $\eps\to 0^+$.

By \eqref{asymp-an}, as $N\to \infty$,
\[
\frac{m(H_{N+1})}{m(H_N)} = \frac{a_{N+1}}{a_N} \sim \Big( \frac{N}{N+1} \Big)^{\frac 1s} \sim 1\, .
\]

Let's consider first the case $s> 1$. For any $\eps >0$, we have,
\[
 \frac{\gamma_\mu(H_{N_\eps+1})}{c_s\cdot m(H_{N_\eps+1})^s}\left[ \frac{m(H_{N_\eps+1})}{m(H_{N_\eps})} \right]^s \le \frac{\gamma_\mu(H_\eps)}{c_s \cdot m(H_\eps)^s} \le \frac{\gamma_\mu(H_{N_\eps})}{c_s\cdot m(H_{N_\eps})^s} \left[ \frac{m(H_{N_\eps})}{m(H_{N_\eps +1})} \right]^s .
\]

Then for $\eps$ small enough, hence $N_\eps$ big enough, the lower and upper bounds above are $\sim 1$. Hence using the sandwich theorem we get the desired asymptotic.

For $s=1$, the same argument leads to 
\[
\frac{\gamma_\mu(H_\eps)}{c_s \cdot m(H_\eps)/\log \frac{1}{m(H_\eps)}} \le \frac{\gamma_\mu(H_{N_\eps})}{c_s \cdot \left( m(H_{N_\eps}) / \log \frac{1}{m(H_{N_\eps})} \right)} \cdot \frac{\left[ m(H_{N_\eps}) / \log \frac{1}{m(H_{N_\eps})}\right]}{\left[ m(H_{N_\eps+1}) / \log \frac{1}{m(H_{N_\eps+1})}\right]} \sim 1 \, ,
\]
and,
\[
\frac{\gamma_\mu(H_\eps)}{c_s \cdot m(H_\eps)/\log \frac{1}{m(H_\eps)}} \ge \frac{\gamma_\mu(H_{N_\eps+1})}{c_s \cdot \left( m(H_{N_\eps+1}) / \log \frac{1}{m(H_{N_\eps+1})} \right)} \cdot \frac{\left[ m(H_{N_\eps+1}) / \log \frac{1}{m(H_{N_\eps+1})}\right]}{\left[ m(H_{N_\eps}) / \log \frac{1}{m(H_{N_\eps})}\right]} \sim 1 \,.
\]
The argument for $s\in (0,1)$ follows along the same lines, hence, concludes the proof.
\end{proof}



\end{document}